\documentclass[10pt]{article}
\usepackage{amssymb,amsmath,amsthm,multirow}
\usepackage{color}
\usepackage[all]{xy}

\newtheorem{remark}{Remark}

\newtheorem{proposition}{Proposition}

\def\F{\mathbb{F}}

\usepackage{authblk}

\title{A Note on a Standard Model for Galois Rings}
 \author[1]{E. Mart\'{\i}nez-Moro\thanks{Partially supported by the Spanish AEI under grant PGC2018-096446-B-C21.}}
 \author[2]{A. Pi\~nera-Nicol\'as\thanks{Partially supported by MINECO-13-MTM2013-45588-C3-1-P.}}
\author[2]{I.F. R\'ua\thanks{
Partially supported by MINECO-13-MTM2013-45588-C3-1-P and Principado de Asturias Grant GRUPIN14-142.}}

\affil[1]{Mathematics Research Institute (IMUVa), Universidad de Valladolid \\
\tt{Edgar.Martinez@uva.es} }

\affil[2]{Departamento de Matem\'aticas, Universidad de Oviedo\\
\tt{apnicolas@uniovi.es, rua@uniovi.es}}
\begin{document}

\maketitle

\begin{abstract} 

In this work we present a standard model for Galois rings based on the standard model of their residual fields, that is, a  a sequence of Galois rings starting with  ${\mathbb Z}_{p^r}$ that coves all the Galois rings with that characteristic ring and such that there is an algorithm producing each member of the sequence whose input is the size of the required ring.\\[2em]



\end{abstract}

\section{Introduction}

During the past decades a great avenue of applications of finite fields have flourish especially on coding theory topics, see for example \cite[Chapters 6,17,18]{Sole} and the references therein. Many of those works require a huge amount of computational work on searching for codes over finite rings and specifically over Galois rings and polynomial rings over Galois rings that are the building blocks of commutative rings \cite{McD}. Nevertheless, up to the knowledge of the authors of this note, there has not been an attempt of building a standard model for dealing with Galois rings computationally in the following sense: a sequence of Galois rings with the same base ring ${\mathbb Z}_{p^r}$ covering all the Galois rings with that characteristic ring and an algorithm that has as input the size of one member of the family and as output a model for that ring  (see Section 2 for a more precise definition). Note for example, that Conway polynomials provide such an algorithmic model for finite fields but the construction of those polynomials is somehow difficult \cite{Conway}. Thus we will follow mainly the standard model proposed by Smit and Lenstra for finite fields \cite{Handbook} to devise a standard model for Galois rings based on the standard model of their residue field. 

The outline of the note is quite simple: Section 1 revises the main points in Smit and Lenstra's  standard model. In Section the standard model for Galois rings is proposed and Section 3 provides some worked examples using  Magma \cite{Magma}.

\section{Standard Model for Finite Fields}

We start with a brief description of the standard model for finite fields. Following \cite[Section 11.7]{Handbook}, given a prime $p$ and a positive integer $n$, an explicit model for a finite field of size $p^n$ is a field whose underlying additive group is $\F_p^n$, so that, if $\{e_0,\ldots,e_{n-1}\}$ is the standard $\F_p$-basis of $\F_p^n$, the field structure on $\F_p^n$ is uniquely determined by the $n^3$ elements $a_{ijk}\in\F_p$ such that $e_i\cdot e_j=\sum_{k=0}^{n-1}a_{ijk}e_k$. We assume that an explicit model is given when we both know the prime $p$ and the $n^3$ structure constants $a_{ijk}$.

Alternatively, we can construct an explicit model by giving an irreducible polynomial $f\in\F_p[x]$ of degree $n$ over $\F_p$. In such case, the use of the basis $\{1,x\ldots,x^{n-1}\}$ of the field $\F_p[x]/(f(x))$ leads us to the explicit model of the field of $p^n$ elements. As can be seen in \cite{Lenstra}, this representation can be converted into the other one by deterministic polynomial time algorithms.

According to \cite[Definition 11.7.7]{Handbook}, an algorithmic model for finite fields is a sequence $(A_q)_q$, with $q$ running over all prime powers, such that $A_q$ is an explicit model for a finite field of size $q$ obtained from an algorithm whose input is $q$. As  can be seen in \cite{HB789}, there is a particular algorithmic model $(S_q)$ such that, for all $q$, it can be computed from an explicit model $A_q$ of a field of size $q$ by a deterministic polynomial time algorithm. This model is known as \textit{a standard model for finite fields}. Its construction as follows.

A finite field $\F_{p^n}$ can be written as the tensor product $\F_{p^n}=\bigotimes_{r|n}\F_{p^{r^k}}$, where $\prod_{r|n}r^k$ is the prime decomposition of $n$. Hence, the standard model of any finite field $\F_{p^n}$ is determined by the knowledge of the standard models of the finite fields $\F_{p^{r^k}}$ with $r$ prime.

As a consequence, two different kind of models are considered: the ones with $r=p$, and the ones with a prime number $r\neq p$. In the first case, the Artin-Schreier construction of cyclic extensions of dimension $p$ of a field of characteristic $p$, see \cite{Jacobson_Basic_Algebra_II}, provides a valid model for the description of the tower of degree $p$ field extensions
$$\F_p\subset\F_{p^{p}}\subset\F_{p^{p^2}}\subset\ldots\subset\F_{p^{p^k}}\subset\ldots$$

More specifically, these extensions are generated by the polynomial $f(x,y)=x^p-1-y\sum_{i=1}^{p-1}x^i\in\F_p[x,y]$. Note that as pointed in \cite{Handbook}, the field $\F_{p^{p^k}}$ is obtained as the quotient of the polynomial ring $\F[x_1,x_2,\ldots,x_k]$ modulo the ideal generated by $\{f(x_1,1),f(x_{i+1},x_i)\,|\,1\leq i\leq k\}$, i.e., $\F_{p^{p^{i+1}}}=\F_{p^{p^i}}(\alpha_{i+1})$, where $x^p-1-\alpha_i\left(\sum_{i=1}^{p-1}x^i\right)$ is the minimal polynomial of $\alpha_{i+1}$ over $\F_{p^{p^i}}$.

When $p\neq r$, the standard model is given by the following construction (see \cite{Handbook}). Let us denote $\textbf{r}=r\cdot\mbox{gcd}(r,2)$; recall that $r$ is a prime number and so $\mathbf{r}=r$, when $r>2$, and $\mathbf{r}=4$, when $r=2$. Let $A_{r,i}$ be the polynomial ring $\mathbb{Z}[x_0,x_1,\ldots,x_i]$ modulo the ideal generating by $\{\sum_{j=0}^{r-1}x_0^{j\textbf{r}/r},x_{k+1}^r-x_k\,|\, 0<k\leq i\}$. Notice that the residue class of the element $x_i$ is a primitive $\textbf{r}r^i$-th root of unity, which will be denoted as $\zeta_{\textbf{r}r^i}$. $A_{r,j}$ can be straightforwardly seen as a subring of $A_{r,i}$ for all $j< i$, thus we will define $A_r=\cup_{i=0}^{\infty}A_{r,i}$. The ring $A_{r,i}$ can be identified with the ring of integers of the number field $\mathbb{Q}(\zeta_{\mathbf{r}r^i})$, whereas $A_r$ can be seen as the ring of integers of $\mathbb{Q}(\zeta_{\textbf{r}r^i}\,:\, i\ge 0)$, whose Galois group is $\mathbb{Z}^*_r$, the group of units of the ring of $r$-adic integers. The torsion subgroup of $\mathbb{Z}_r^*$ will be denoted by $\Delta_r$, so that $A_r\cong \Delta_r\times\Gamma$, where $\Gamma\cong\mathbb{Z}_r^+$. The subgroup $\Delta_r$ is isomorphic to $(\mathbb{Z}/\textbf{r}\mathbb{Z})^*$ and so, it is cyclic of order $\varphi(\textbf{r})$, where $\varphi$ denotes the Euler's totient function. 

Let us denote by $B_{r,i}$ the subring of $A_{r,i}$ fixed by the automorphisms of $\Delta_r$, that is, $B_{r,i}=\{a\in A_{r,i}\,|\, \delta(a)=a,\mbox{ for all }\delta\in \Delta_r\}$. From now on, $B_r$ will denote the union $\cup_{i=0}^{\infty}B_{r,i}$.

 For the prime number $p$, the set of prime ideals $\mathfrak{p}$ of $B_r$ containing $p$ will be denoted by $S_{p,r}$. It can be proved that this set has cardinality $r^l$, where $l=\mbox{ord}_r(p^{\varphi(r)}-1)/(\textbf{r}^2/r)$, i.e., $r^l$ is the largest power of $r$ dividing $(p^{\varphi(r)}-1)/(\mathbf{r}^2/r)$. Moreover, for each $\mathfrak{p}$ in $S_{p,r}$ there exists a unique system $(a_{\mathfrak{p},j})_{0\leq j< lr}$ of integers in $\{0,1,\ldots,p-1\}$ such that the ideal $\mathfrak{p}$ is generated as a module over $B_r$ by the prime $p$ together with $\{\eta_{r,k+1,i}-a_{\mathfrak{p},i+kr}\,|\,0\leq k<l,0\leq i<r\}$, where $\eta_{r,k,i}=\sum_{\delta\in\Delta_r}\delta(\zeta_{\mathbf{r}r^k}^{1+i\mathbf{r}r^{k-1}})$ is an element of $B_{r,k}$.

It is possible to define a total ordering on the set $S_{p,r}$ as follows. If $\mathfrak{p}$ and $\mathfrak{q}$ are two different prime ideals containing $p$, we say that $\mathfrak{p}<\mathfrak{q}$ if there exists $h\in\{0,1,\ldots,lr-1\}$ such that $a_{\mathfrak{p},j}=a_{\mathfrak{q},j}$ for all $j< h$, and $a_{\mathfrak{p},h}<a_{\mathfrak{q},h}$, i.e., the ideals in $S_{p,r}$ are ordered according to the lexicographical order on the integers $a_{\mathfrak{p},j}$ which describe them. The smallest element of $S_{p,r}$ in this ordering is denoted by $\mathfrak{p}_{p,r}$. The prime ideal $\mathfrak{p}_{p,r}\cap B_{r,k}$ will be denoted as $\mathfrak{p}_{p,r,k}$. If $\overline{\alpha}_{p,r,k}$ is the residue class of $\eta_{p,r,k}$ modulo $\mathfrak{p}_{p,r,k}$, then it can be proved that, for all $k\geq 0$, the field $\F_p(\overline{\alpha}_{r,l+k,0})$ has cardinality $p^{r^k}$. This field provides the standard model for a field of cardinality $p^{r^k}$ as defined in \cite{Handbook}.

Summarizing, the defining element of the finite field $\F_{p^{r^k}}$ is the residue class (by a certain prefixed prime ideal $\mathfrak{p}_{p,r,k}$) of the Gauss period of the $\mathbf{r}r^{k+l}$-th root of unity with respect to the subgroup $\Gamma$ of $\mathbb{Z}_r^*$. Consequently, the smallest element of $S_{p,r}$ uniquely determines the minimal polynomial of such defining element over $\F_p$.

\section{Standard Model for Galois Rings}

Given a prime number $p$ and two positive integers $n$ and $r$, the Galois Ring $\mbox{GR}(p^n,r)$ is defined as the quotient of the polynomial ring $\mathbb{Z}[x]$ modulo the ideal generated by $p^n$ and a monic polynomial $h(x)\in\mathbb{Z}[x]$ of degree $r$ which is irreducible modulo $(p)=p\mathbb{Z}$, that is, $\mbox{GR}(p^n,r)\cong\mathbb{Z}[x]/(p^n,h(x))$. Galois Rings are local and their ideals form the chain
$$\mbox{GR}(p^n,r)\supseteq p\,\mbox{GR}(p^n,r)\supseteq\ldots\supseteq p^i\,\mbox{GR}(p^n,r)\supseteq\ldots\supseteq p^n\,\mbox{GR}(p^n,r)=\{0\}.$$

\noindent Notice that the ideal $p\,\mbox{GR}(p^n,r)$ is maximal and $\mbox{GR}(p^n,r)$ is an unramified extension of the ring $\mathbb{Z}/p^n\mathbb{Z}$, that is, the maximal ideal of $\mbox{GR}(p^n,r)$ is generated by the maximal ideal of $\mathbb{Z}/p^n\mathbb{Z}$. The quotient $\mbox{GR}(p^n,r)/p\,\mbox{GR}(p^n,r)$ is isomorphic to the finite field $\F_{p^r}$, and it is called the residual field of $\mbox{GR}(p^n,r)$.

% In an alternative way, $\mbox{GR}(p^n,r)$ can be seen as a Galois extension of the ring $\mathbb{Z}/p^n\mathbb{Z}$ of degree $r$, see \cite{McD}. So, following \cite{Janusz}, $\mbox{GR}(p^n,r)$ is the unique strongly separable $\mathbb{Z}/p^n\mathbb{Z}$-algebra without proper idempotens and of rank $r$ over $\mathbb{Z}/p^n\mathbb{Z}$. That is, it is finitely generated, projective and separable over $\mathbb{Z}/p^n\mathbb{Z}$. In fact, Galois Rings can be characterized as the only rings without proper idempotents that are of prime power characteristic and are separable over the subring generated by the identity element. 
%
% We can consider two kind of natural maps between the rings $\mbox{GR}(p^n,r)$ for different values of $n$ and $r$. First of all, there is an inclusion of the ring $\mbox{GR}(p^n,r)$ into the ring $\mbox{GR}(p^n,rd)$ for any positive integer $d$.  On the other hand, there exists a natural projection
% $$\pi_n:\mbox{GR}(p^n,r)\rightarrow\mbox{GR}(p^{n-1},r),$$
%
% \noindent whose kernel is the ideal $p^{n-1}\mbox{GR}(p^n,r)$. The collection $\{\mbox{GR}(p^n,r);\pi_n\}$, for a fixed $r$, is an inverse mapping system. The inverse limit $\mbox{projlim}\{\mbox{GR}(p^n,r);\pi_n\}$ can be identified with the ring of $p$-adic integers $\mathbb{Z}_p$, for $r=1$, or with a finite separable extension of $\mathbb{Z}_p$, for $r\neq 1$. 

On the other hand, Theorem XVI.8 of \cite{McD} provides a decomposition of the Galois Ring $\mbox{GR}(p^n,r)$ as a tensor product of Galois Rings, according to the prime power decomposition of $r$. So, if $r=\prod_{i=1}^s r_i$, then

\begin{equation}\label{eq:Galois_tensor}
\mbox{GR}(p^n,r)\cong\bigotimes_{i=1}^s\mbox{GR}(p^n,r_i).
\end{equation}

Thus, as for finite fields, we have to deal with two different cases: the case with $r=p^k$ and the case with $(r,p)=1$.

First of all, we will study the case $\mbox{GR}(p^n,r^k)$ with $r$ a prime number different from $p$. Let us consider the ring $B_{r,k}$ described in the previous section. As we have seen, the standard model for finite fields is based on a choice of a prime ideal $\mathfrak{p}_{p,r,k}$ of $B_{r,k}$ such that the quotient $B_{r,k}/\mathfrak{p}_{p,r,k}\cong\F_p(\overline{\alpha}_{r,k+l,0})\cong\F_{p^{r^k}}$, where $\overline{\alpha}_{r,k+l,0}$ is the residue class of the Gauss period $\eta_{r,l+k,0}=\sum_{\delta\in\Delta_r}\delta(\zeta_{\mathbf{r}r^k})$ modulo the ideal $\frak{p}_{p,r,k}$. The powers of such a prime ideal will provide the standard model for Galois Rings. We claim that, for all $n\geq 1$, $B_{r,k}/\mathfrak{p}_{p,r,k}^n\cong\mbox{GR}(p^n,r^k)$.

\begin{proposition}
Let $A_{r,k}$ be the ring of integers of the cyclotomic field $\mathbb{Q}(\zeta_{\mathbf{r}r^k})$, where $\mathbf{r}=(r,2)$. Let $\Delta_r$ be the torsion group of the group of units $\mathbb{Z}_r^*$ of $r$-adic integers, isomorphic to the Galois group of the extension $\mathbb{Q}(\zeta_{\mathbf{r}r^k}\,|\,i\ge 0)\,|\,\mathbb{Q}$. Let $B_{r,k}$ the subring of $A_{r,k}$ fixed by $\Delta_r$ and let $\mathfrak{p}$ be a prime ideal of $B_{r,k}$ containing the prime integer $p$ (in particular, $\mathfrak{p}=\mathfrak{p}_{p,r,k}$). Then $B_{r,k}/\mathfrak{p}^n\cong\mbox{GR}(p^n,r^k)$.
\end{proposition}

\begin{proof}
Following \cite{Clark}, it suffices to prove that $\mbox{Rad}(B_{r,k}/\mathfrak{p}^n)\cong p\left(B_{r,k}/\mathfrak{p}^n\right)$. Notice that $B_{r,k}$ is a Dedekind domain and so, $B_{r,k}/\mathfrak{p}^n$ is a local ring whose maximal ideal is $\mathfrak{p}/\mathfrak{p}^n$. It is well known that $p B_{r,k}= \prod_{\mathfrak{q}\in \tilde{S}_{p,r}}\mathfrak{q}^{e_{\mathfrak{q}}}$, where $\tilde{S}_{p,r}$ is the set of prime ideals of $B_{r,k}$ containing $p$ and $e_\mathfrak{q}$ is the ramification index of $\mathfrak{q}$ over $\mathbb{Q}$, which will be equal to 1. Otherwise, $p\in\mathfrak{q}^{e_{\mathfrak{q}}}$ with $e_{\mathfrak{q}}>1$. If $\mathfrak{s}$ is a prime ideal of $A_{r,k}$ containing the prime ideal $\mathfrak{q}$, then $p\in\mathfrak{q}^{e_\mathfrak{q}}\subseteq\mathfrak{s}^{e_{\mathfrak{q}}}$. But this is impossible, since $A_{r,k}$ is unramified over $\mathbb{Q}$, see \cite{Milne}. Thus, we have
$$p\left(B_{r,k}/\mathfrak{p}^n\right)= \left(pB_{r,k}+\mathfrak{p}^n\right)/\mathfrak{p}^n\cong pB_{r,k}/\left(pB_{r,k}\cap\mathfrak{p}^n\right)\cong \prod_{\mathfrak{q}\in S_{p,r}}\mathfrak{q}/\mathfrak{q}^{a_{\mathfrak{q}}}\cong \mathfrak{p}/\mathfrak{p}^n,$$
\noindent where $a_{\mathfrak{p}}=n$ and $a_{\mathfrak{q}}=1$ for $\mathfrak{q}\neq\mathfrak{p}$. Since $B_{r,k}/\mathfrak{p}^n$ is a local ring of characteristic $p^n$ with residual field isomorphic to $\F_{p^{r^k}}$, it follows that $B_{r,k}/\mathfrak{p}^n\cong\mbox{GR}(p^n,r^k)$.
\end{proof}

\begin{remark}
Notice that, since all prime ideals $\mathfrak{p}$ of $B_{r,k}$ containing $p$ are unramified over $\mathbb{Q}$, the quotients $B_{r,k}/\mathfrak{p}^n$ are Galois rings for all $n\ge 1$. In general, for an arbitrary non Dedekind domain $R$,  the quotient $R/\mathfrak{p}^n$ is a finite local chain ring whose structure is given in \cite[Theorem XVII.1]{McD}. 
\end{remark}

So, we have that, for all $n\ge 1$,
 $$\left(B_{r,k}/\mathfrak{p}_{p,r,k}^n\right)/ p\left(B_{r,k}/\mathfrak{p}_{p,r,k}^n\right)\cong B_{r,k}/\mathfrak{p}_{p,r,k}\cong\F_{p^{r^k}}.$$
 
 \noindent Let us denote by $f(x)$ the minimal polynomial of the Gauss period $\eta_{r,l+k,0}$ over $\mathbb{Q}[x]$. Notice that, since $\eta_{r,l+k,0}\in B_{r,k}$, we can ensure that $f(x)\in\mathbb{Z}[x]$. {Let us consider the $p$-adic field $\mathbb{Q}_p$. Since $f(x)\in\mathbb{Q}[x]$ is separable, this polynomial can be factorized in $\mathbb{Q}_p[x]$ into the product of $m$ pairwise different monic irreducible factors. So, we have that
 \begin{equation}\label{eq:Zp_factorization}
  f(x)=\prod_{i=1}^{m}f_i(x)\in{\mathbb{Q}_p[x]}.
\end{equation}}

\noindent{From Hensel's Lemma \cite[Th. 6.1.2]{Gouveia}, it follows that the residue modulo $\mathfrak{p}$ of each factor $f_i(x)$ is a power of some irreducible polynomial in $\F_p[x]$, i.e, $\overline{f_i}(x)=g_i(x)^{e_i}$. Since $A_{r,k}$ is unramified over $\mathbb{Q}$, so are their completions with respect to the $p$-adic absolute value, and so, $e_i=1$ for all $1\leq i\leq m$. Thus, the polynomials $\overline{f}_i(x)\in\F_p[x]$ are irreducible over $\F_p$ for all $1\leq i\leq m$.}
%  Since there is a one to one correspondence between the irreducible factors of $f(x)\in\mathbb{Q}_p[x]$ and the prime ideals of $B_{r,k}$ containing $p$, the reduction of $f$ modulo $p$  factorizes into the product of $r^l$ \textcolor{red}{monic} irreducible, and possibly repeated, polynomials of degree $r^k$ each.
%  
% \begin{equation}\label{eq:p_factorization}
% \overline{f}(x)=\prod_{i=1}^{r^l}\overline{f}_i(x)\in\mathbb{F}_p[x].
% \end{equation}

% \noindent\textcolor{red}{Thus, $m=r^l$ and the factors $f_i(x)\in\mathbb{Q}_p[x]$ can be obtained from the lifting of (\ref{eq:p_factorization}) to $\mathbb{Q}_p[x]$.} 
 
% \noindent Since $f(x)$ can be seen in $\mathbb{Q}_p[x]$, this factorization can be lifted, using the $p$-adic version of Hensel's Lemma, to obtain
%  
% \begin{equation}\label{eq:Zp_factorization}
%   f(x)=\prod_{i=1}^{r^l}f_i(x)\in\textcolor{red}{\mathbb{Z}_p[x]}.
% \end{equation}

% \noindent Notice that, \textcolor{red}{for each $i$, the reduction $\overline{f}_i(x)$ modulo $p$ is monic and irreducible over $\F_p$, so $f_i(x)$ is irreducible in $\mathbb{Q}_p[x]$, see Corollary 5.3.8 of \cite{Gouveia}. Moreover, since $f(x)\in\mathbb{Q}[x]$ is separable, the factors $f_i(x)\in\mathbb{Q}_p[x]$ must be pairwise different. Thus, factorization (\ref{eq:Zp_factorization}) must coincide with the factorization of $f(x)$ into irreducible polynomials in $\mathbb{Q}_p[x]$.}
%   
{Let $f_1(x)\in\mathbb{Q}_p[x]$ be the factor of $f(x)$ such that $f_1(\eta_{r,l+k,0})=0$. Its reduction modulo $\mathfrak{p}_{p,r,k}$, $\overline{f}_1(x)\in\F_p[x]$, is irreducible, and so, it is the minimal polynomial of the residue of $\eta_{r,l+k,0}$ modulo $\mathfrak{p}_{p,r,k}$, denoted by $\overline{\alpha}_{r,l+k,0}$. Let us denote by $K$ the finite extension $\mathbb{Q}_p(\eta_{r,l+k,0})$, which is clearly unramified. The valuation ring of $K$, denoted by $\mathfrak{O}_K$, is a discrete valuation ring whose elements are integral over $\mathbb{Z}_p$. This ring is local, with maximal ideal $\mathfrak{p}\cong\mathfrak{p}_{p,r,k}$, and $\mathfrak{O}_K/\mathfrak{p}\cong B_{r,k}/\mathfrak{p}_{p,r,k}\cong\F_{p^{r^k}}$. Moreover, $\mathfrak{O}_K$ is a Dedekind domain, and the quotient $\mathfrak{O}_k/\mathfrak{p}^n$ is a strongly separable algebra of rank $r^k$ over the subring generated by the identity and so, isomorphic to $\mbox{GR}(p^n,r^k)$, see \cite{Janusz}.} Thus, 
$$B_{r,k}/\mathfrak{p}_{p,r,k}^n\cong\mbox{GR}(p^n,r^k)\cong\mathfrak{O}_k/\mathfrak{p}^n\cong\mathbb{Z}[x]/(f_1(x),p^n).$$

\noindent So, the quotient $\mathbb{Z}[x]/(f_1(x),p^n)$ gives the standard model for the Galois Ring $\mbox{GR}(p^n,r^k)$.

{Let us suppose that the index of $\mathbb{Z}[\eta_{r,l+k,0}]$ in $B_{r,k}$ is relatively prime with $p$. Thus, starting from the standard model of its residual field, we can give the standard model for the Galois Ring $\mbox{GR}(p^n,r^k)$ by means of the integral version of Hensel's Lemma \cite[Th. 1.4.3]{Bini}.} In such case, we know that the standard model for the residual field $\F_{r^k}$ is given by $\F_p(\overline{\alpha}_{p,r,k})$. Let $f(x)\in\mathbb{Z}[x]$ be the minimal polynomial of $\eta_{r,l+k,0}$ over $\mathbb{Q}[x]$, which can be calculated using theoretical techniques, \cite{Gurak}, or computational ones, \cite{Cohen}. The integral version of Hensel's Lemma allows the lifting of the factorization $\overline{f}=\prod_{i=1}^{r^l}\overline{f}_i(x)\in\mathbb{Z}/p\mathbb{Z}[x]$ to $\mathbb{Z}/p^n\mathbb{Z}[x]$. Thus, we can find basic irreducible polynomials $g_i(x)\in\mathbb{Z}/p^n\mathbb{Z}[x]$ such that 
$$f(x)\equiv\prod_{i=1}^{r^l}g_i(x)\ (\mbox{mod }p^n).$$

% \begin{theorem}[Hensel's Lemma]\label{lemma:Hensel_integral}
% Let $p$ be a prime and $k\leq 1$ a positive integer; suppose that $f(x)$, $g(x)$ and $h(x)$ are monic polynomials in $\mathbb{Z}[x]$ such that $f(x)$ and $g(x)$ are relatively prime modulo $p$ and 
% $$f(x)\equiv g(x)h(x)\ (\mbox{mod }p^k).$$
%
% \noindent Then, there exist polynomials $g_1(x),h_1(x)\in\left(\mathbb{Z}/p^{k+1}\mathbb{Z}\right)[x]$, uniquely determined and relatively prime modulo $p$, such that
% $$\begin{array}{ccl}
% g(x)&\equiv& g_1(x)\ (\mbox{mod }p^k)\\
% h(x)&\equiv& h_1(x)\ (\mbox{mod }p^k)\\
% f(x)&\equiv& g_1(x)\,h_1(x)\ (\mbox{mod }p^{k+1})\\
% \end{array}$$    
%
% \end{theorem}

\noindent {So, there will be a unique polynomial, say $g_1(x)$, such that $\overline{g}_1(\overline{\alpha}_{p,r,k})=0$. Since $\mbox{GR}(p^n,r^k)\cong\mathbb{Z}[x]/(p^n,g_1(x))$, this polynomial gives the standard model for a Galois Ring of rank $r^k$ over the subring generated by its identity element. Notice that this construction remains valid as long as the multiplicity of the polynomial $\overline{g}_1(x)$ in the factorization of $\overline{f}(x)\in\mathbb{Z}/p\mathbb{Z}[x]$ is equal to 1. }

% \noindent Notice that, since the ramification index of an ideal $\mathfrak{p}\subseteq B_{r,k}$, containing $p$, over $\mathbb{Q}$ is one, there will exist $n_0\in\mathbb{N}$ such that this factorization does not contain repeated factors for $n\geq n_0$. Thus, for $n=n_0$, there will be a unique polynomial, say $g_1(x)$, such that $\overline{g}_1(\overline{\alpha}_{p,r,k})=0$. Since $\mbox{GR}(p^n,r^k)\cong\mathbb{Z}[x]/(p^n,g_1(x))$. This polynomial gives the standard model for a Galois Ring of rank $r^k$ over the subring generated by its identity element.

It only remains the case $r=p^k$. Notice that the residual field of $\mbox{GR}(p^n,p^k)$ is isomorphic to $\F_{p^{p^k}}$ and so, there exists a standard model for its quotient field given by an irreducible polynomial $f(x)\in\F_p[x]$ of degree $p^k$. Recall that $f(x)$ can be found using the Artin-Schreier theory of cyclic extensions of dimension $p$ over a field of characteristic $p$. This polynomial remains irreducible when it is embedded into $\mathbb{Z}[x]$ and so, it is basic irreducible. Thus, $\mbox{GR}(p^n,p^k)\cong\mathbb{Z}[x]/(p^n,f(x))$, which gives the standard model for this kind of Galois Rings.

{Let us denote by $K_{r^k}$ the (unique) unramified extension of $\mathbb{Q}_p$ of degree $r^k$, which is the splitting field of the polynomial $x^{p^{r^k}}-x\in\mathbb{Q}_p[x]$. Clearly, $K_{r^k}=\mathbb{Q}_p(\eta_{r,k+l,0})$ and $K_{r^k}\subset K_{r^{k+1}}$ for all $k> 0$. For each prime number $r$, let us denote $\mathbb{Q}_{p,r}=\mathbb{Q}_p(\eta_{r,l+1,0},\eta_{r,l+2,0},\ldots)=\bigcup_{k>0}K_{r^k}$.}

{Following \cite{Handbook}, let $\overline{\mathbb{Q}}_p$ be the tensor product, over $\mathbb{Q}_p$, of the rings $\mathbb{Q}_{p,r}$, with $r$ ranging over the set of all prime numbers. It is clear that $\overline{\mathbb{Q}}_p\subset\mathbb{Q}_p^{{un}}$, where $\mathbb{Q}_p^{{un}}$ is the unramified closure of $\mathbb{Q}_p$. Moreover, we have that $\overline{Q}_p=\mathbb{Q}_p(\eta_{r,l+k,0}\,|\, r\mbox{ prime},\, k>0)$, and, for each $k> 0$, the element $\eta_{r,l+k,0}$ is algebraic of degree $r^k$ over $\mathbb{Q}_p$.}

{Given an element $s\in\mathbb{Q}/\mathbb{Z}$, there exists a unique set of integers $(c_{r,k})$, with $r$ ranging over the set of prime numbers and $k$ over $\mathbb{Z}_{>0}$, with $c_{r,k}\in \{0,1,\ldots,r-1\}$, such that $s\equiv\sum_{r,k}c_{r,k}/r^k\,(\mbox{mod }\mathbb{Z})$. Notice that the sum is finite, since the integers $c_{r,k}$ are equal to zero for all but finitely many pairs $(r,k)$. With this notation, the element $\epsilon_s$ is defined as the finite product $\prod_{r,k}\eta_{r,l+k,0}^{c_{r,k}}\in\overline{Q}_p$. It is straightforward to prove that the system $\{\epsilon_s\,|\, s\in\mathbb{Q}/\mathbb{Z}\}$ is a $\mathbb{Q}_p$-basis of $\overline{Q}_p$.}
 
{Given two integers $n\geq 1$ and $m\geq 1$, consider the $\mathbb{Z}_p$-module of $\overline{Q}_p$ generated by the basis $\{ \epsilon_0,\epsilon_{1/m}\ldots\epsilon_{(m-1)/m}\}$, which can be identified with the valuation ring $\mathfrak{O}_{K_m}$ of the unramified extension $K_m$ of $\mathbb{Q}_p$. As we have seen,
$$\mbox{GR}(p^n,m)\cong\mathfrak{O}_{K_m}/\mathfrak{p}^n\cong\langle \epsilon_0,\epsilon_{1/m}\ldots\epsilon_{(m-1)/m}\rangle_{\mathbb{Z}_p}\bigotimes\mathbb{Z}/p^n\mathbb{Z}. $$}

{Let $A$ be a free $\mathbb{Z}/p^n\mathbb{Z}$-module of rank $m$, and let $\psi:A\rightarrow\mathfrak{O}_{K_m}\bigotimes\mathbb{Z}/p^n\mathbb{Z}$ be the unique $\mathbb{Z}/p^n\mathbb{Z}$-module isomorphism which sends, for $0\leq i\leq m-1$, the element $a_i$ of the $\mathbb{Z}/p^n\mathbb{Z}$-basis of $A$ to the element $\epsilon_{i/m}$ of the basis of $\mathfrak{O}_{K_m}\bigotimes\mathbb{Z}/p^n\mathbb{Z}$. The standard model for the Galois Ring $\mbox{GR}(p^n,m)$ is the explicit model is given by the free $\mathbb{Z}/p^n\mathbb{Z}$-module $A$ endowed with the product $u\cdot v=\psi^{-1}(\psi(u)\cdot\psi(v))$, for all $u,v\in A$. }

\section{Some Examples}

In this section, we present some examples of the construction of the standard model for several Galois Rings. We first analyse the case
with $r=3$ and $p\neq 3$ and, later on, the case $r=p=3$. All calculations were made with Magma, \cite{Magma}. 

Let $r=3$, and let $\zeta_{3^{i+1}}$ be a $3^{i+1}$-th primitive root of unity. For all $i\ge 0$, the ring $A_{3,i}$ can be identified with the ring of integers of the cyclotomic field $\mathbb{Q}(\zeta_{3^{i+1}})$. So, $A_3$ can be seen as the ring of integers of the field $\mathbb{Q}(\zeta_{3^i}:i\ge 0)$. The Galois group of the field extension $\mathbb{Q}(\zeta_{3^i}:i\ge 0)\,|\,\mathbb{Q}$ is the group of $3$-adic units $\mathbb{Z}_3^*\cong (\mathbb{Z}/3\mathbb{Z})^*\times\mathbb{Z}_3^+$. Thus, $\Delta_3\cong(\mathbb{Z}/3\mathbb{Z})^*$, which is a cyclic group of order 2. For each $i\ge 0$, recall that $B_{3,i}$ is the subring of $A_{3,i}$ fixed by the automorphisms of $\Delta_3$, and that $B_3=\cup_{i=0}^{\infty}B_{3,i}$.

Let us consider $p=17$. Since $l=\mbox{ord}_3\left((17^2-1)/3\right)=1$, the set $S_{17,3}$ of prime ideals of $B_3$ containing $p=17$ has three elements: 
$$\begin{array}{rcl}
\mathfrak{p}_1&=&\langle 17,\,\eta_{3,1,0}-7,\,\eta_{3,1,1}-14,\,\eta_{3,1,2}-13\rangle_{B_3},\\ 
   \mathfrak{p}_2&=&\langle 17,\,\eta_{3,1,0}-13,\,\eta_{3,1,1}-7,\,\eta_{3,1,2}-14\rangle_{B_3},\\
   \mathfrak{p}_3&=&\langle 17,\,\,\eta_{3,1,0}-14,\eta_{3,1,1}-13,\,\eta_{3,1,2}-7\rangle_{B_3},\\
   \end{array}$$
  
\noindent where $\eta_{3,1,j}=\zeta_{3^2}^{1+3j}+\zeta_{3^2}^{-(1+3j)}$ for $0\leq j\leq 2$. Following \cite{Handbook}, according to the ordering induced in $S_{17,3}$, the smallest element is $\mathfrak{p}_1$. So, $\mathfrak{p}_{17,3}=\langle 17,\,\eta_{3,1,0}-14,\,\eta_{3,1,1}-13,\,\eta_{3,1,2}-7\rangle_{B_3}$.

As in the previous section, let us denote by $\overline{\alpha}_{17,3,k}$ the image of $\eta_{3,k+1,0}=\zeta_{3^{k+2}}+\zeta_{3^{k+2}}^{-1}$ in the field $B_3/(\mathfrak{p}_{17,3}\cap B_{3,k+1})\cong\F_{17^{3^k}}$. The minimal polynomial of $\eta_{3,k+1,0}$ over $\mathbb{Q}$ is known to be, see \cite{Gurak_p3}, $f_k(x)=x^{3^{k+1}}F_k(x^{-1})$, where

$$F_k(x)=x^{3^{k+1}}+\sum_{n=0}^{[3^{k+1}/2]}(-1)^n\frac{3^{k+1}}{3^{k+1}-n} {3^{k+1}-n\choose{n}}x^{2n}.$$

\noindent The polynomial $f_k(x)\in\mathbb{Z}[x]$ will factorize in $\F_{17}[x]$ into 3 factors $g_{k,i}(x)$, $i=1,\cdots,3$, of degree $3^k$ each. For $k=0$, $f_0(x)=x^3-3x+1=(x-14)(x-13)(x-7)$, which splits over $B_3/(\mathfrak{p}_{17,3}\cap B_{3,1})\cong\F_{17}$. For $k\ge 1$, from \cite{Gurak_p3}, we have that $\eta_{3,k+1,0}^3-3\eta_{3,k+1,0}=\eta_{3,k,0}$, so, taking the projection modulo $\mathfrak{p}_{17,3}$, $\overline{\alpha}_{17,3,k+1}$ will be a root of the polynomial $g_{k,1}(x)=x^3-3x-\overline{\alpha}_{17,3,k}\in\F_{17}(\overline{\alpha}_{17,3,k})$, which is irreducible over $\F_{17}(\overline{\alpha}_{17,3,k})[x]$. From Capelli's Lemma, \cite{Ayad}, we obtain the irreducibility of the composition $g_{k,1}(x)=(x^3-3x)^{\circ k}-\overline{\alpha}_{17,3,0}\in\F_{17}[x]$, which gives the minimal polynomial of the element $\overline{\alpha}_{17,3,k}$ defining the standard model for a field of size $17^{3^k}$. A similar argument can be applied to the Gauss periods $\eta_{3,k+1,1}$ and $\eta_{3,k+1,3}$ to obtain the polynomials $g_{k+1,2}(x)$ and $g_{k+1,2}(x)$. Thus, for each $k\geq 0$, we get the following factorization of $f_k(x)$ in $\F_{17}[x]$ into pairwise irreducible factors:
$$f_k(x)\equiv g_{k,1}(x)g_{k,2}(x)g_{k,3}(x)\ (\mbox{mod } 17),$$

% For each $k\ge 0$, let us take $f_k(x)=x^3-3x-\alpha_{17,3,k}\in\F_{17}(\alpha_{17,3,k})$. If $k=0$, $\alpha_{17,3,0}=-1$ and the polynomial $f_0(x)=x^3-3x+1=(x-14)(x-13)(x-7)$ splits over $\F_{17}[x]$. Otherwise, $f_k(x)$ is irreducible over $\F_{17}(\alpha_{17,3,k})[x]$ and $\alpha_{17,3,k+1}$ is defined as a root of $f_k(x)$. From Capelli's Lemma \cite{Ayad}, we obtain the irreducibility of the composition $f(x)=(x^3-3x)^{\circ k}-\alpha_{17,3,0}\in\F_{17}[x]$, which gives the standard model for a field of size $17^{3^k}$.
%
Once the factorization of $f_k(x)$ in $\F_{17}[x]$ is known, the standard model for $\mbox{GR}(17^n,3^k)$ is given by the integral version of Hensel's Lemma, \cite[Th. 1.4.3]{Bini}. This information is summarized in Table \ref{Table:Ejemplos_r3p17}. Notice that, for each row, the polynomial on the column with $n=1$ is irreducible over $\F_{17}[x]$. Moreover, it is the reduction modulo 17 of the other ones, which are basic irreducible. The last column, polynomials over $\mathbb{Z}_{17}[x]$, is determined by the polynomials of the first column applying the $p$-adic version of Hensel's Lemma, see \cite[Th. 6.1.2]{Gouveia}. All calculations were made with a precision of $17^{10}$.

\begin{table}[h]
\begin{tabular}{c|p{2.5cm}p{2.5cm}p{2.5cm}cp{2.5cm}|}
k&$n=1$&$n=2$&$n=3$&$\cdots$&$\mathbb{Z}_{17}$\\ \hline
 $0$ & $x+10$ &$x+214$&$x+1659$ &$\cdots$&$x+907573721136$ \\ \hline
$1$ & $x^3+14x+10$& $x^3+286x+214$ &$x^3+4910x+1659$ &$\cdots$ &$x^3-3x+907573721136$\\ \hline
$2$ & $x^9+8x^7+10x^5+4x^3+9x+10$& $x^9+280x^7+27x^5+259x^3+9x+214$& $x^9+4904x^7+27x^5+4883x^3+9x+1659$ &$\cdots$&$x^9-9x^7+27x^5-30x^3+9x+907573721136$\\ \hline
$3$ & $x^{27}+7x^{25}+x^{23}+x^{21}+8x^{19}+15x^{17}+16x^7+7x^5+3x^3+7x+10$ & $x^{27}+262x^{25}+35x^{23}+35x^{21}+280x^{19}+49x^{17}+119x^{15}+255x^{13}+187x^{11}+187x^9+254x^7+143x^5+241x^3+262x+214$ & $x^{27}+4886x^{25}+324x^{23}+2636x^{21}+569x^{19}+2072x^{17}+986x^{15}+3434x^{13}+4233x^{11}+765x^9+1410x^7+2455x^5+819x^3+4886x+1659$& $\cdots$ &$x^{27}-27x^{25}+324x^{23}-2277x^{21}+10395x^{19}-32319x^{17}+69768x^{15}-104652x^{13}+107406x^{11}-72930x^9+30888x^7-7371x^5+819x^3        -27x-907573721136$\\ \hline
$\vdots$&$\cdots$&$\cdots$&$\cdots$&$\cdots$&$\cdots$\\ \hline
$k$&$(x^3-3x)^{\circ_{k-1}}+10$&$(x^3-3x)^{\circ_{k-1}}+214$&$(x^3-3x)^{\circ_{k-1}}+1659$&$\cdots$&$(x^3-3x)^{\circ_{k-1}}+907573721136$\\  \hline
\end{tabular}\caption{Standard models for $\mbox{GR}(17^n,3^k)$}\label{Table:Ejemplos_r3p17}
\end{table}

For the next example, let $r=5$ and let $\zeta_{5^{i+1}}$ be a $5^{i+1}$-th primitive root of unity. As for $r=3$, the ring $A_{5,i}$ can be identified with the ring of integers of the cyclotomic field $\mathbb{Q}(\zeta_{5^{i+1}})$ and $A_5$ can be seen as the ring of integers of the field $\mathbb{Q}(\zeta_{5^i}\,:\,i\ge 0)$. The Galois group of the field extension $\mathbb{Q}(\zeta_{5^i}\,:\,i\ge 0)\,|\,\mathbb{Q}$ is the group of $5$-adic units $\mathbb{Z}_5^*\cong (\mathbb{Z}/5\mathbb{Z})^*\times\mathbb{Z}_5^+$. Thus, $\Delta_5\cong(\mathbb{Z}/5\mathbb{Z})^*$, which is a cyclic group of order 4. For each $i\ge 0$, recall that $B_{5,i}$ is the subring of $A_{5,i}$ fixed by the automorphisms of $\Delta_5$, and that $B_5=\cup_{i=0}^{\infty}B_{5,i}$.

Let us consider $p=7$. Since $l=\mbox{ord}_5\left((7^4-1)/5\right)=1$, the set $S_{7,5}$ of prime ideals of $B_5$ containing $p=7$ has five elements: 
$$\begin{array}{rcl}
  \mathfrak{p}_1&=&\langle 7,\,\eta_{5,1,0}-1,\eta_{5,1,1}-3,\,\eta_{5,1,2}-5,\,\eta_{5,1,3}-2,\,\eta_{5,1,4}-3\rangle_{B_5},\\
  \mathfrak{p}_2&=&\langle 7,\,\eta_{5,1,0}-2,\eta_{5,1,1}-3,\,\eta_{5,1,2}-1,\,\eta_{5,1,3}-3,\,\eta_{5,1,4}-5\rangle_{B_5},\\
    \mathfrak{p}_3&=&\langle 7,\,\eta_{5,1,0}-3,\eta_{5,1,1}-1,\,\eta_{5,1,2}-3,\,\eta_{5,1,3}-5,\,\eta_{5,1,4}-2\rangle_{B_5},\\
     \mathfrak{p}_4&=&\langle 7,\,\eta_{5,1,0}-3,\eta_{5,1,1}-5,\,\eta_{5,1,2}-2,\,\eta_{5,1,3}-3,\,\eta_{5,1,4}-1\rangle_{B_5},\\
    \mathfrak{p}_5&=&\langle 7,\,\eta_{5,1,0}-5,\eta_{5,1,1}-2,\,\eta_{5,1,2}-3,\,\eta_{5,1,3}-1,\,\eta_{5,1,4}-3\rangle_{B_5},\\ 
  \end{array}$$

\noindent where $\eta_{5,1,j}=\sum_{\delta\in\Delta_5}\delta(\zeta_{5^2}^{1+5j})=\zeta_{5^2}^{1+5j}+\zeta_{5^2}^{7+10j}+\zeta_{5^2}^{-(1+5j)}+\zeta_{5^2}^{-(7+10j)}$ for $0\leq j\leq 4$. Following \cite{Handbook}, according to the ordering induced in $S_{7,5}$, the smallest element is $\mathfrak{p}_1$. So, $\mathfrak{p}_{7,5}=\langle 7,\,\eta_{5,1,0}-1,\eta_{5,1,1}-3,\,\eta_{5,1,2}-5,\,\eta_{5,1,3}-2,\,\eta_{5,1,4}-3\rangle_{B_5}$.

% Let us denote by $\overline{\alpha}_{7,5,k}$ the image of $\eta_{5,k+1,0}=\sum_{\delta\in\Delta_5}\delta(\zeta_{5^{k+2}}^{1+5j})$ in the field $B_5/(\mathfrak{p}_{7,5}\cap B_{5,k+1})\cong\F_{7^{5^k}}$. 

{As can be seen in \cite{Gurak}, no recurrence formula is known for the minimal polynomial of $\eta_{5,1,0}$ over $\mathbb{Q}$. Using Magma, we found that this minimal polynomial is }
$$f_0(x)=x^5 - 10x^3 + 5x^2 + 10x + 1,$$

\noindent which factorizes in $\F_7[x]$ as
$$f_0(x)\equiv (x-1)(x-2)(x-3)^2(x-5)\ (\mbox{mod }7).$$

\noindent Notice that the factor $(x-3)$ appears with multiplicity 2. However, {the embedding of $f_0(x)$ into the ring $\mathbb{Z}_7[x]$ splits, for a precision of $7^{10}$, into the product of the following five pairwise distinct irreducible factors.}
$$ f_0(x)= (x + 89288611)(x-136787418)(x -55262245)(x - 109033102)(x - 70681095).$$

\noindent Moreover, the projection of $f_0(x)$ over $\mathbb{Z}/7^n\mathbb{Z}[x]$ splits for $n\ge 3$. 

For $k\ge 1$, the polynomial $f_k(x)$ factorizes over $\F_{7}[x]$ into five different irreducible factors of degree $5^k$:

$$f_k(x)\equiv g_{k,1}(x)g_{k,2}(x)g_{k,3}(x)g_{k,4}(x)g_{k,5}(x)\ (\mbox{mod } 7).$$

\noindent The standard model for $\mbox{GR}(7^n,5^k)$ is given by the quotient $(\mathbb{Z}/7^n\mathbb{Z})[x]/(G_{n,k,1}(x))$, where $G_{n,k,1}(x)$ is the lift of $g_{1,k}(x)$ modulo $7^n$ obtained from the factorization of $f_k(x)$ using the integral version of Hensel's Lemma, \cite[Th. 1.4.3]{Bini}. These polynomials, for $k\leq 2$ and $n\leq 3$, can be found in Table \ref{Table:Ejemplos_r5p7}. The last column, polynomials over $\mathbb{Z}_{7}[x]$, is calculated from the first one applying the $p$-adic version of Hensel's Lemma, see \cite[Th. 6.1.2]{Gouveia}. All calculations were made with a precision of $7^{10}$.

\begin{table}[h]
\begin{tabular}{c|p{2.5cm}p{2.5cm}p{2.5cm}cp{2.5cm}|}
k&$n=1$&$n=2$&$n=3$&$\cdots$&$\mathbb{Z}_{7}$\\ \hline
 $0$ & $x+6$ &$x+27$&$x+223$ &$\cdots$&$x + 89288611$ \\ \hline
$1$ & $x^5+4x^3+3x^2+2x+6$& $x^5 +39x^3 +10x^2 + 23x +27 $ &$x^5 + 333x^3 +108x^2 + 121x +223 $ &$\cdots$ &$x^5 - 10x^3 - 118986592x^2 - 70930221x + 89288611$\\ \hline
$2$ & $x^{25} + 6x^{23} + 3x^{21} + 6x^{19} + 3x^{18} + 4x^{17} + 4x^{16} + 5x^{15} + 2x^{14}
        + 3x^{13} + 2x^{12} + 2x^{11} + 6x^9 + 2x^8 + 5x^7 + 4x^6 + 6x^5 + x^4
        + 2x^3 + 5x^2 + 2x + 6$& 
        $ x^{25} + 48x^{23} + 45x^{21} + 20x^{19} + 31x^{18} + 18x^{17} + 25x^{16} + 5x^{15} +
        30x^{14} + 17x^{13} + 2x^{12} + 44x^{11} + 35x^{10} + 13x^9 + 30x^8 +
        26x^7 + 4x^6 + 41x^5 + 36x^4 + 9x^3 + 12x^2 + 44x + 27$ &
        $x^{25} + 293x^{23} + 339x^{21} + 69x^{19} + 178x^{18} + 214x^{17} + 319x^{16} +
        152x^{15} + 275x^{14} + 311x^{13} + 296x^{12} + 142x^{11} + 280x^{10} +
        258x^9 + 324x^8 + 222x^7 + 102x^6 + 237x^5 + 183x^4 + 205x^3 +
        12x^2 + 289x + 223$&
        $\cdots$&$x^{25} - 50x^{23} + 1025x^{21} - 11250x^{19} - 110679405x^{18} - 27514217x^{17} -
        79903246x^{16} + 137649825x^{15} + 16072226x^{14} + 132000431x^{13} -
        25843382x^{12} + 50890709x^{11} - 7926107x^{10} - 125486292x^9 -
        71853031x^8 - 57656706x^7 + 113389042x^6 - 40325244x^5 -
        101821768x^4 - 8793629x^3 + 63002252x^2 + 38678341x + 89288611$\\ \hline
\end{tabular}\caption{Standard models for $\mbox{GR}(7^n,5^k)$}\label{Table:Ejemplos_r5p7}
\end{table}

Finally, let us consider $r=p=3$. So, we will obtain standard models for the Galois Rings $\mbox{GR}(3^n,3^k)$, with $n\ge 1$ and $k\ge 1$. These constructions will be made using, as it was referred in the previous section, the Artin-Schreier theory of cyclic extensions of dimension $p$ over fields of characteristic $p$.

We start by the construction of standard models for this tower of degree 3 field extensions
$$\F_3\subset\F_{3^3}\subset\F_{3^{3^2}}\subset\ldots\subset\F_{3^{3^k}}\subset\ldots$$

\noindent It can be seen that the polynomial $f_k(x)=x^3-x+\alpha_k^{-1}$ is irreducible over $\F_{3}(\alpha_k)[x]$, where $\alpha_{k}=1+\beta_{k}^{-1}$, for $k\ge 1$, $\alpha_0=1$, and $\beta_{k}\in\F_{3^{3^{k}}}$ is a root of $f_{k-1}(x)$. So, we can construct a sequence $(g_k)_{k\ge 0}$ of irreducible polynomials over $\F_3[x]$ such that, for each $k$, $g_k$ gives the standard model of the field $\F_{3^{3^k}}$. The embedding these polynomials into $\mathbb{Z}[x]$ gives a sequence of basic irreducible polynomials. The standard model of $\mbox{GR}(3^n,3^k)$, for $n>1$, is given by the quotient $\mathbb{Z}[x]/(p^n,f(x))$. These models are summarized in Table \ref{Table:Ejemplos_r3p3}.

\begin{table}[h]
\begin{tabular}{|p{2.5cm}|p{2.5cm}|p{2.5cm}|}
$k=1$&$k=2$&$k=3$\\ \hline

$x^3+2x^2+2x+2$& $x^9+2x^8+x^7+x^5+x^3+2x+2$& $x^{27}+2x^{26}+2x^{24}+x^{23}+2x^{22}+2x^{21}+2x^{18}+x^{16}+x^{15}+2x^{14}+2x^{13}+2x^{12}+x^{10}+2x^9+x^8+x^7+x^6+2x^5+x^3+x^2+2x+2$\\ \hline
%$\vdots$&$\cdots$&$\cdots$&$\cdots$&$\cdots$\\ \hline
%$k$&$(x^3-3x)^{\circ_{k-1}}+10$&$(x^3-3x)^{\circ_{k-1}}+214$&$(x^3-3x)^{\circ_{k-1}}+1659$&$\cdots$\\  \hline
\end{tabular}\caption{Standard models for $\mbox{GR}(3^n,3^k)$}\label{Table:Ejemplos_r3p3}
  
\end{table}

\bibliography{ring-semisimple}

\def\cprime{$'$}
\begin{thebibliography}{10}

\bibitem{Ayad}
M.~Ayad and D.~L. McQuillan.
\newblock Irreducibility of the iterates of a quadratic polynomial over a
  field.
\newblock {\em Acta Arith.}, 93(1):87--97, 2000.

\bibitem{Bini}
G.~Bini and F.~Flamini.
\newblock {\em {Finite commutative rings and their applications}}.
\newblock Kluwer Academic Publishers, Boston, MA, 2002.

\bibitem{Magma}
W.~Bosma, J.~Cannon, and C.~Playoust.
\newblock The {M}agma algebra system. {I}. {T}he user language.
\newblock {\em J. Symbolic Comput.}, 24(3-4):235--265, 1997.
\newblock Computational algebra and number theory (London, 1993).

\bibitem{Clark}
W.~E. Clark.
\newblock A coefficient ring for finite non-commutative rings.
\newblock {\em Proc. Amer. Math. Soc.}, 33:25--28, 1972.

\bibitem{Cohen}
H.~Cohen.
\newblock {\em A course in computational algebraic number theory}, volume 138
  of {\em Graduate Texts in Mathematics}.
\newblock Springer-Verlag, Berlin, 1993.

\bibitem{HB789}
B.~de~Smit and H.~W. Lenstra.
\newblock Standard models for finite fields.

\bibitem{Gouveia}
F.~Q. Gouv\^{e}a.
\newblock {\em {$p$}-adic numbers}.
\newblock Universitext. Springer-Verlag, Berlin, second edition, 1997.
\newblock An introduction.

\bibitem{Gurak_p3}
S.~Gurak.
\newblock Minimal polynomials for {G}auss periods with {$f=2$}.
\newblock {\em Acta Arith.}, 121(3):233--257, 2006.

\bibitem{Gurak}
S.~Gurak.
\newblock On the minimal polynomial of {G}auss periods for prime powers.
\newblock {\em Math. Comp.}, 75(256):2021--2035, 2006.

\bibitem{Conway}
L.~S. Heath and N.~A. Loehr.
\newblock New algorithms for generating {C}onway polynomials over finite
  fields.
\newblock {\em J. Symbolic Comput.}, 38(2):1003--1024, 2004.

\bibitem{Sole}
W.~C. {Huffman}, J.-L. {Kim}, and P.~{Sol\'e}, editors.
\newblock {\em {Concise encyclopedia of coding theory}}.
\newblock Boca Raton, FL: CRC Press, 2021.

\bibitem{Jacobson_Basic_Algebra_II}
N.~Jacobson.
\newblock {\em Basic algebra. {II}}.
\newblock W. H. Freeman and Company, New York, second edition, 1989.

\bibitem{Janusz}
G.~J. Janusz.
\newblock Separable algebras over commutative rings.
\newblock {\em Trans. Amer. Math. Soc.}, 122:461--479, 1966.

\bibitem{Lenstra}
H.~W. Lenstra, Jr.
\newblock Finding isomorphisms between finite fields.
\newblock {\em Math. Comp.}, 56(193):329--347, 1991.

\bibitem{McD}
B.~R. McDonald.
\newblock {\em {Finite rings with identity}}.
\newblock Marcel Dekker, Inc., New York, 1974.
\newblock Pure and Applied Mathematics, Vol. 28.

\bibitem{Milne}
J.~S. Milne.
\newblock Algebraic number theory (v3.08), 2020.
\newblock Avalaible at www.jmilne.org/math/.

\bibitem{Handbook}
G.~L. Mullen, editor.
\newblock {\em Handbook of finite fields}.
\newblock Discrete Mathematics and its Applications (Boca Raton). CRC Press,
  Boca Raton, FL, 2013.

\end{thebibliography}
\bibliographystyle{abbrv}

%\begin{thebibliography}{10}

% \end{thebibliography}

\end{document}